\documentclass{amsart}
\usepackage{amsfonts}
\usepackage{amssymb}
\usepackage{amsrefs}
\usepackage{enumerate}
\usepackage{hyperref}
\usepackage{cleveref}
\usepackage[labelformat=simple]{subfig}
\usepackage{csvsimple}
\usepackage{tikz}
\usepackage{pgfplots}
\usepackage[acronym]{glossaries}

\newacronym{srg}{SRG}{\textit{scaled relative graph}}
\newcommand{\notation}[1]{\textcolor{orange!80!black}{\emph{#1}}}

\newcommand{\abs}[1]{\ensuremath{\left\vert#1\right\vert}}
\newcommand{\cfunof}[1]{\ensuremath{\left\{#1\right\}}}
\newcommand{\C}{\ensuremath{\mathbb{C}}}
\newcommand{\cl}[1]{\ensuremath{\mathrm{cl}\,{#1}}}
\newcommand{\diag}[1]{\ensuremath{\mathrm{diag}\left({#1}\right)}}
\newcommand{\funof}[1]{\ensuremath{\left(#1\right)}}
\newcommand{\F}{\ensuremath{\mathbb{F}}}
\newcommand{\graph}[1]{\ensuremath{\mathrm{gra}\,{#1}}}
\newcommand{\hilb}{\ensuremath{\mathcal{H}}}
\newcommand{\hull}[1]{\ensuremath{\mathrm{co}\left(#1\right)}}
\newcommand{\hullbk}[1]{\ensuremath{\mathrm{co}_{\mathrm{Be\text{-}}\!\mathrm{Kl}}\left(#1\right)}}
\newcommand{\hvect}[1]{\ensuremath{\mathbf{#1}}}
\newcommand{\imag}[1]{\ensuremath{\mathrm{Im}\left(#1\right)}}
\newcommand{\inner}[2]{\ensuremath{\left\langle #1,#2\right\rangle}}
\newcommand{\innerm}[1]{\ensuremath{\mathrm{m}\left(#1\right)}}
\newcommand{\ltwo}{\ensuremath{\ell^2\left(\mathbb{N}\right)}}
\newcommand{\Ltwo}{\ensuremath{\mathcal{L}^2\left(\mathbb{R}\right)}}
\newcommand{\norm}[1]{\ensuremath{\left\Vert #1 \right\Vert}}
\newcommand{\real}[1]{\ensuremath{\mathrm{Re}\left(#1\right)}}

\newcommand{\ri}[1]{\ensuremath{\mathrm{d}_\mathrm{s}\left(#1\right)}}
\newcommand{\ro}[1]{\ensuremath{\mathrm{d}_\mathrm{l}\left(#1\right)}}
\newcommand{\R}{\ensuremath{\mathbb{R}}}
\newcommand{\spectrum}[1]{\ensuremath{\sigma\left(#1\right)}}
\newcommand{\spectrumap}[1]{\ensuremath{\sigma_{\mathrm{ap}}\left(#1\right)}}
\newcommand{\srg}[1]{\ensuremath{\mathrm{SRG}\left(#1\right)}}
\newcommand{\srgc}[1]{\ensuremath{\mathrm{cl}\;\mathrm{SRG}\left(#1\right)}}

\newcommand{\tm}{\ensuremath{\left(t\right)}}
\newcommand{\w}{\ensuremath{\left(\omega\right)}}
\newcommand{\W}[1]{\ensuremath{\mathrm{W}\left(#1\right)}}
\newcommand{\Wc}[1]{\ensuremath{\mathrm{cl}\;\mathrm{W}\left(#1\right)}}

\numberwithin{equation}{section}
\crefformat{equation}{(#2#1#3)}
\crefmultiformat{equation}{(#2#1#3)}%
{ and~(#2#1#3)}{, (#2#1#3)}{ and~(#2#1#3)}

\newtheorem{theorem}{Theorem}
\newtheorem{lemma}{Lemma}
\theoremstyle{remark}
\newtheorem{remark}{Remark}
\newtheorem{example}{Example}

\tikzset{>=stealth}
\usetikzlibrary{calc}
\usetikzlibrary{external}
\pgfplotsset{compat=1.13}
\makeatother
\pgfdeclarelayer{bg}
\pgfsetlayers{bg,main}

\hypersetup{
    colorlinks = true,
    linkcolor = orange!80!black,
    anchorcolor = blue,
    citecolor = orange!80!black,
    filecolor = blue,
    urlcolor = orange!80!black
    }

\makeatletter

\def\hyper@x#1,#2\relax{#1}
\def\hyper@y#1,#2\relax{#2}
\def\hyper@coords#1{#1}

\newif\ifhyper@vertical

\def\hyper@computer#1#2{%
  \edef\hyper@toscan{(#1)}
  \tikz@scan@one@point\hyper@coords\hyper@toscan
  \edef\hyper@sx{\the\pgf@x}
  \edef\hyper@sy{\the\pgf@y}
  \edef\hyper@toscan{(#2)}
  \tikz@scan@one@point\hyper@coords\hyper@toscan
  \edef\hyper@ex{\the\pgf@x}
  \edef\hyper@ey{\the\pgf@y}
  \pgfmathsetmacro{\hyper@mx}{(\hyper@ex + \hyper@sx)/2}
  \pgfmathsetmacro{\hyper@my}{(\hyper@ey + \hyper@sy)/2}
  \pgfmathsetmacro{\hyper@dx}{\hyper@ex - \hyper@sx}
  \pgfmathparse{\hyper@dx == 0 ? "\noexpand\hyper@verticaltrue" : "\noexpand\hyper@verticalfalse"}
  \pgfmathresult
  \ifhyper@vertical
  \edef\hyper@cmd{-- (\tikztotarget)}
  \else
  \pgfmathsetmacro{\hyper@dy}{\hyper@ey - \hyper@sy}
  \pgfmathsetmacro{\hyper@t}{\hyper@my/\hyper@dx}
  \pgfmathsetmacro{\hyper@cx}{\hyper@mx + \hyper@t * \hyper@dy}
  \pgfmathsetmacro{\hyper@radius}{veclen(\hyper@cx - \hyper@sx, \hyper@sy)}
  \pgfmathsetmacro{\hyper@sangle}{180 - atan2(\hyper@sy,\hyper@cx-\hyper@sx)}
  \pgfmathsetmacro{\hyper@eangle}{180 - atan2(\hyper@ey,\hyper@cx-\hyper@ex)}
  \edef\hyper@cmd{arc[radius=\hyper@radius pt, start angle=\hyper@sangle, end angle=\hyper@eangle]}
  \fi
}

\def\hyper@disc@computer#1#2{%
  \edef\hyper@toscan{(#1)}
  \tikz@scan@one@point\hyper@coords\hyper@toscan
  \edef\hyper@sx{\the\pgf@x}
  \edef\hyper@sy{\the\pgf@y}
  \edef\hyper@toscan{(#2)}
  \tikz@scan@one@point\hyper@coords\hyper@toscan
  \edef\hyper@ex{\the\pgf@x}
  \edef\hyper@ey{\the\pgf@y}
  \pgfmathsetmacro{\hyper@det}{\hyper@sx * \hyper@ey - \hyper@sy * \hyper@ex}
  \pgfmathparse{\hyper@det == 0 ? "\noexpand\hyper@verticaltrue" : "\noexpand\hyper@verticalfalse"}
  \pgfmathresult
  \ifhyper@vertical
  \edef\hyper@cmd{-- (\tikztotarget)}
  \else
  \pgfmathsetmacro{\hyper@mx}{(\hyper@ex + \hyper@sx)/2}
  \pgfmathsetmacro{\hyper@my}{(\hyper@ey + \hyper@sy)/2}
  \pgfmathsetmacro{\hyper@dx}{\hyper@ex - \hyper@sx}
  \pgfmathsetmacro{\hyper@dy}{\hyper@ey - \hyper@sy}
  \pgfmathsetmacro{\hyper@dradius}{\pgfkeysvalueof{/tikz/hyperbolic disc radius}}
  \pgfmathsetmacro{\hyper@t}{((\hyper@dradius)^2 - \hyper@sx * \hyper@ex - \hyper@sy * \hyper@ey)/(2 * (\hyper@sx * \hyper@ey - \hyper@sy * \hyper@ex))}
  \pgfmathsetmacro{\hyper@radius}{sqrt((\hyper@t)^2 + .25) * veclen(\hyper@dx,\hyper@dy)}
  \pgfmathsetmacro{\hyper@cx}{\hyper@mx + \hyper@t * \hyper@dy}
  \pgfmathsetmacro{\hyper@cy}{\hyper@my - \hyper@t * \hyper@dx}
  \pgfmathsetmacro{\hyper@sangle}{atan2(\hyper@sy-\hyper@cy,\hyper@sx - \hyper@cx)}
  \pgfmathsetmacro{\hyper@eangle}{atan2(\hyper@ey-\hyper@cy,\hyper@ex - \hyper@cx)}
  \pgfmathsetmacro{\hyper@eangle}{\hyper@eangle > \hyper@sangle + 180 ? \hyper@eangle - 360 : \hyper@eangle}
  \edef\hyper@cmd{arc[radius=\hyper@radius pt, start angle=\hyper@sangle, end angle=\hyper@eangle]}
\fi
}

\def\hyper@plane@tangent#1#2{%
  \edef\hyper@toscan{(#1)}
  \tikz@scan@one@point\hyper@coords\hyper@toscan
  \edef\hyper@sx{\the\pgf@x}
  \edef\hyper@sy{\the\pgf@y}
  \edef\hyper@toscan{(#2)}
  \tikz@scan@one@point\hyper@coords\hyper@toscan
  \edef\hyper@ex{\the\pgf@x}
  \edef\hyper@ey{\the\pgf@y}
  \pgfmathsetmacro{\hyper@ex}{\hyper@ex - \hyper@sx}
  \pgfmathsetmacro{\hyper@ey}{\hyper@ey - \hyper@sy}
  \pgfmathparse{\hyper@ex == 0 ? "\noexpand\hyper@verticaltrue" : "\noexpand\hyper@verticalfalse"}
  \pgfmathresult
  \ifhyper@vertical
  \pgfmathsetmacro{\hyper@d}{\hyper@ey/1cm}
  \pgfmathsetmacro{\hyper@radius}{\hyper@sy * exp(\hyper@d) - \hyper@sy}
  \edef\hyper@cmd{-- ++(0,\hyper@radius pt)}
  \else
  \pgfmathsetmacro{\hyper@d}{\hyper@ex > 0 ? veclen(\hyper@ex,\hyper@ey) : -veclen(\hyper@ex,\hyper@ey)}
  \pgfmathsetmacro{\hyper@radius}{abs(\hyper@sy * \hyper@d / \hyper@ex)}
  \pgfmathsetmacro{\hyper@sangle}{90 + atan(\hyper@ey/\hyper@ex)}
  \pgfkeysgetvalue{/tikz/hyperbolic plane target angle}{\hyper@eangle}
  \ifx\hyper@eangle\pgfutil@empty
  \pgfmathsetmacro{\hyper@d}{\hyper@d/1cm}
  \pgfmathsetmacro{\hyper@ey}{\hyper@ey/1cm}
  \pgfmathsetmacro{\hyper@tanhd}{tanh(\hyper@d)}
  \pgfmathsetmacro{\hyper@eangle}{acos((\hyper@d * \hyper@tanhd - \hyper@ey)/(\hyper@d - \hyper@ey * \hyper@tanhd))}
  \fi
  \edef\hyper@cmd{arc[radius=\hyper@radius pt, start angle=\hyper@sangle, end angle=\hyper@eangle]}
\fi
}

\tikzset{%
  hyperbolic disc radius/.initial={1cm},
  hyperbolic plane/.style={
    to path={
      \pgfextra{\hyper@computer\tikztostart\tikztotarget}
      \hyper@cmd
    }
  },
  hyperbolic plane tangent/.style={
    to path={
      \pgfextra{\hyper@plane@tangent\tikztostart\tikztotarget}
      \hyper@cmd
    }
  },
  hyperbolic disc/.style={
    to path={
      \pgfextra{\hyper@disc@computer\tikztostart\tikztotarget}
      \hyper@cmd
    }
  },
  hyperbolic plane target angle/.initial={},
}

\pgfmathdeclarefunction{bkx}{2}{%
\begingroup
 \pgfmathparse{1-(2/(1+#1*#1+#2*#2))}
 \pgfmath@smuggleone\pgfmathresult%
\endgroup}

\pgfmathdeclarefunction{bky}{2}{%
\begingroup
 \pgfmathparse{-(2*#1/(1+#1*#1+#2*#2))}
 \pgfmath@smuggleone\pgfmathresult%
\endgroup}

\begin{document}

\title{The Scaled Relative Graph of a Linear Operator}
\author{Richard Pates}

\thanks{Department of Automatic Control, Lund University, Box 118, SE-221 00, Lund, Sweden. \textit{E-mail:} \href{mailto:richard.pates@control.lth.se}{\texttt{richard.pates@control.lth.se}}}

\thanks{The author is a member of the ELLIIT Strategic Research Area at Lund University. This work was supported by the ELLIIT Strategic Research Area. This project has received funding from VR 2016-04764, SSF RIT15-0091 and ERC grant agreement No 834142.}

\date{June 10, 2011 and, in revised form, ****.}

\maketitle

\begin{abstract}
The \gls{srg} of an operator is a subset of the complex plane. It captures several salient features of an operator, such as contractiveness, and can be used to reveal the geometric nature of many of the inequality based arguments used in the convergence analyses of fixed point iterations. In this paper we show that the \gls{srg} of a linear operator can be determined from the numerical range of a closely related linear operator. Furthermore we demonstrate that the \gls{srg} of a linear operator has a range of spectral and convexity properties, and satisfies an analogue of Hildebrant's theorem.
\end{abstract}

\glsresetall
\section{Introduction}
The \gls{srg} was introduced by Ryu, Hannah and Yin in \cite{RHY19} as a geometric tool for the modular analysis of operators. The \gls{srg} of an operator is a subset of the complex plane that captures a number of important features of the operator, such as whether or not it is contractive. The \glspl{srg} of simpler operators can be combined in an intuitive graphical manner to bound the \glspl{srg} of the operators resulting from their algebraic composition. These rules for combining operators can be used as geometric analogues of the inequalities typically used in, for example, the convergence proofs of fixed point iterations. This has been used to give a unified geometric treatment of the convergence rates of a wide range of algorithms, including gradient descent, Douglas-Rachford splitting and the method of alternating projections.

The promise of the \gls{srg} extends far beyond the analysis of algorithms from convex optimization. As already noted in \cite{CFS21}, the modular fashion in which the \gls{srg} can be manipulated makes it an ideal candidate for dynamical system analysis, and the authors additionally give preliminary results connecting the \gls{srg} to classical tools from control theory. In order to unlock this potential, a better understanding of how to determine the \gls{srg} of an operator is required. For example, even the question of how to determine the \gls{srg} when the operator is a square matrix with real entries has only been fully resolved in the case that the matrix is normal, or of dimension 2 \cite{HRY20}.

Our primary motivation is to better understand the geometry of the \gls{srg}, building our intuition from the finite dimensional linear case, where the operators in question are matrices. However, from a theoretical perspective, the results from the matrix case can be pushed through to the case of linear operators on Hilbert spaces with little to no changes. Since such operators are relevant in a wide range of applications, particularly in the study of differential equations, this is the setting we will consider. Our main result is to show that the \gls{srg} of a linear operator can be determined from the numerical range of a closely related linear operator. This allows much of the machinery that has been developed to understand the numerical range to be applied in the \gls{srg} setting. We use this to show that the \gls{srg}, like the numerical range, has a range of convexity and spectral properties, and satsifies an analogue of Hildebrant's theorem \cite{Hil66}. Despite these similarities, the convexity properties of the \gls{srg} are rooted in hyperbolic geometry, and its spectral properties capture information about the approximate point spectrum rather than the spectrum.

\Cref{sec:2} introduces the relevant concepts from the theory of linear operators and hyperbolic geometry, and also reviews the definition of the \gls{srg} and known results on the \gls{srg} of a matrix. In \cref{sec:3} we relate the \gls{srg} to the numerical range. \Cref{sec:31} establishes the connections in the case of complex Hilbert spaces. In this subsection we also characterise the spectral and convexity properties of the \gls{srg}, derive the analogue of Hildebrant's theorem, and show how to plot the boundary of the \gls{srg} of an operator defined either by a matrix or a linear differential equation. Finally in \cref{sec:32} we show how to determine the \gls{srg} of a linear operator on a real Hilbert space using the results from \cref{sec:31}.

\section{Notation and preliminaries}\label{sec:2}

\subsection{Basic notation}

Throughout \F{} will denote either the \hypertarget{realfield}{\notation{real field}}, \R{}, or the \hypertarget{complexfield}{\notation{complex field}}, \C{}. When speaking geometrically we will also refer to \C{} as the \hypertarget{complexplane}{\notation{complex plane}}.  The \hypertarget{complexconjugate}{\notation{complex conjugate}} of $z\in\C$ will be denoted by $\bar{z}$. A set $s$ is said to be \hypertarget{convex}{\notation{convex}} if $ts_1+\funof{1-t}s_2\in{}s$ for all $0\leq{}t\leq{}1$ and $s_1,s_2\in{}s$, and the \hypertarget{closure}{\notation{closure}} of $s$ is denoted by $\cl{s}$. Furthermore, the \hypertarget{convexhull}{\notation{convex hull}} $\hull{s}$ is defined to be the smallest \hyperlink{convex}{convex} set containing $s$, and the \hypertarget{boundary}{\notation{boundary}} of a set $s\subseteq\C$ will be denoted by $\partial{}s$ ($\cl{s}\cap\cl{\cfunof{\C\setminus{}s}}$). We will overload notation as appropriate to apply to sets, for example $\overline{\cfunof{z_1,z_2}}$ will denote the set \cfunof{\overline{z}_1,\overline{z}_2}, and more generally $h\funof{s}=\cfunof{h\funof{z}:z\in{}s}$.

\subsection{Operators on Hilbert spaces}\label{sec:11}

$\hilb$ denotes a \hypertarget{Hilbertspace}{\notation{Hilbert space}} over the \hyperlink{realfield}{field} \F{}, equipped with an \hypertarget{innerproduct}{\notation{inner product}} $\langle\cdot{,}\cdot\rangle:\hilb\times{}\hilb\rightarrow{}\F{}$ which defines a \hypertarget{norm}{\notation{norm}} $\norm{\cdot{}}=\sqrt{\langle\cdot{,}\cdot\rangle}$. $T:\hilb\rightarrow{}\hilb$ will be called a \hypertarget{boundedlinearoperator}{\notation{linear operator}} if it is linear, and $\sup\cfunof{\norm{T\hvect{x}}:\hvect{x}\in\hilb,\norm{\hvect{x}}=1}<\infty$. The \hypertarget{identityoperator}{\notation{identity operator}} will be denoted by $I$ ($I\hvect{x}=\hvect{x}$ for all $\hvect{x}\in\hilb$).

To illustrate our results we will primarily consider the cases that
\begin{enumerate}
  \item $\hilb$ is $\R^n$, equipped with the \hyperlink{innerproduct}{inner product} $\inner{\hvect{y}}{\hvect{x}}=\hvect{x}^\mathsf{T}\hvect{y}$;
  \item $\hilb$ is $\C^n$, equipped with the \hyperlink{innerproduct}{inner product} $\inner{\hvect{y}}{\hvect{x}}=\overline{\hvect{x}}^\mathsf{T}\hvect{y}$;
\end{enumerate}
in which case the \hyperlink{boundedlinearoperator}{linear operators} correspond to the square matrices with entries in \R{} and \C{} respectively. We will also consider the \hyperlink{Hilbertspace}{Hilbert space} of complex valued Lebesgue \hypertarget{squareintegrablefunction}{\notation{square integrable functions}} $\Ltwo$, with \hyperlink{innerproduct}{inner product}
\[
\inner{\hvect{y}}{\hvect{x}}=\int_{\R}\overline{\hvect{x}\tm{}}^\mathsf{T}\hvect{y}\tm\,dt,\;\hvect{x},\hvect{y}\in\Ltwo,
\]
and the \hyperlink{Hilbertspace}{Hilbert space} of complex valued \hypertarget{squaresummable}{\notation{square summable sequences}} $\ltwo$, with \hyperlink{innerproduct}{inner product}
\[
\inner{\hvect{y}}{\hvect{x}}=\sum_{j\in\mathbb{N}}\overline{\hvect{x}}_j\hvect{y}_j,\;\hvect{x},\hvect{y}\in\ltwo.
\]
We define the \hypertarget{graph}{\notation{graph}} of a \hyperlink{boundedlinearoperator}{linear operator} $T$ as
\[
\graph{T}=\cfunof{\funof{\hvect{x},T\hvect{x}}:\hvect{x}\in\hilb},
\]
and denote the \hypertarget{adjoint}{\notation{adjoint}} of $T$ as $T^*$ ($\inner{T\hvect{x}}{\hvect{y}}=\inner{\hvect{x}}{T^*\hvect{y}}$ for all $\hvect{x},\hvect{y}\in\hilb$). $T$ is said to be \hypertarget{boundedlyinvertible}{\notation{invertible}} if there is a \hyperlink{boundedlinearoperator}{linear operator} $S$ such that $TS=ST=I$, and we denote this inverse as $T^{-1}$. The \hypertarget{extendedspectrum}{\notation{spectrum}} of $T$ defined to be the subset of the \hyperlink{complexplane}{complex plane}
\[
\spectrum{T}=\cfunof{z:z\in\F,\funof{T-zI}\text{ is not \hyperlink{boundedlyinvertible}{invertible}}}.
\]
We additionally say that $\lambda\in\spectrum{T}$ is in the \hypertarget{approximatepointspectrum}{\notation{approximate point spectrum}} ($\lambda\in\spectrumap{T}$) if there exist a sequence of unit vectors such that $\lim_{n\rightarrow{}\infty}\norm{\funof{T-\lambda{}I}\hvect{x}_n}=0$. In the matrix case $\spectrum{T}=\spectrumap{T}$, and $\lambda$ is an eigenvalue of $T$ if and only if $\lambda\in\spectrum{T}$.

\subsection{The scaled relative graph}\label{sec:21}

We define the \gls{srg} of a \hyperlink{boundedlinearoperator}{linear operator} $T$ to be the subset of the \hyperlink{complexplane}{complex plane}
\[
\srg{T}=\cfunof{\frac{\norm{\hvect{y}}}{\norm{\hvect{x}}}\exp\funof{\pm{}i\arccos\funof{\frac{\real{\inner{\hvect{y}}{\hvect{x}}}}{\norm{\hvect{y}}\norm{\hvect{x}}}}}:\hvect{x}\in\hilb,\hvect{y}=T\hvect{x},\norm{\hvect{x}}=1}.
\]
For \hyperlink{boundedlinearoperator}{linear operators} this definition coincides with the more general definition of the \gls{srg} from \cite{RHY19}. It follows from the definition of the \gls{srg} that $T$ is contractive if and only if $\srg{T}$ is contained in the closed unit disk.

The \gls{srg} captures some of the geometric features of the input-output pairs of the operator. Recall that the angle $\theta$ between $\hvect{x}\in\hilb$ and $\hvect{y}\in\hilb$ is typically defined through
\begin{equation}\label{eq:polar}
\cos\theta=\frac{\real{\inner{\hvect{y}}{\hvect{x}}}}{\norm{\hvect{y}}\norm{\hvect{x}}}.
\end{equation}
The \gls{srg} is then the union of the `polar representations' of the input-output pairs $\funof{\hvect{x},\hvect{y}}\in\graph{T}$, in which the magnitude is given by the ratio between the \hyperlink{norm}{norms} of the output and input, and the argument the angle between the input and output.


\subsection{The Beltrami-Klein mapping}\label{sec:23}

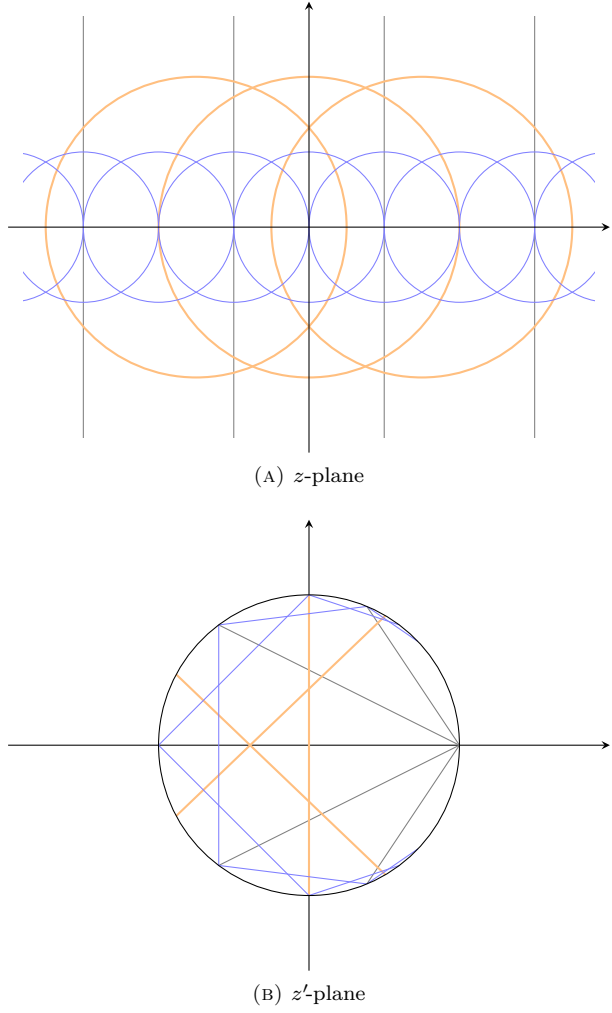
\begin{figure}

\subfloat[$z$-plane\label{local_label}]
{
\begin{tikzpicture}[scale=2]

\begin{scope}
\clip (-1.9,-1.4) rectangle (1.9,1.4);
\foreach \x in {-1.5,-.5,...,1.5}
  \draw[black!50] (\x,-2) -- (\x,2);

\foreach \x in {-.75,0,.75}
  \draw[orange!50,thick] (\x,0) circle (1);

\foreach \x in {-2,-1.5,...,2}
  \draw[blue!50] (\x,0) circle (.5);

\end{scope}

\draw[->] (-2,0) -- (2,0);
\draw[->] (0,-1.5) -- (0,1.5);

\end{tikzpicture}
}

\subfloat[$z'\!$-plane\label{local_label}]
{
\begin{tikzpicture}[scale=2]

\draw[->] (-2,0) -- (2,0);
\draw[->] (0,-1.5) -- (0,1.5);

\foreach \x in {-1.5,-.5,...,1.5}
  \draw[black!50] ({bkx(\x,0)},{bky(\x,0)})--(1,0);

\foreach \x in {-.75,0,.75}
  \draw[orange!50,thick] ({bkx(\x-1,0)},{bky(\x-1,0)})--({bkx(\x+1,0)},{bky(\x+1,0)});

\foreach \x in {-2,-1.5,...,2}
  \draw[blue!50] ({bkx(\x-.5,0)},{bky(\x-.5,0)})--({bkx(\x+.5,0)},{bky(\x+.5,0)});

\draw (0,0) circle (1);

\end{tikzpicture}
}
\caption{\label{fig:bk} Illustration of the \protect\hyperlink{beltramikleinmapping}{Beltrami-Klein mapping}. The generalised circles centred on the real axis in {\footnotesize \textsc{({\tiny A})}} are mapped by $z'=f\funof{z}$ to the chords of the unit circle in {\footnotesize \textsc{({\tiny B})}}. The chords are sent back to their corresponding generalised circles by $z=g\funof{z'}$.}
\end{figure}

The \hypertarget{beltramikleinmapping}{\notation{Beltrami-Klein mapping}} is a tool from two dimensional hyperbolic geometry. Its importance in the context of the \gls{srg} was first recognised in \cite{HRY20}, where it was used in the construction of the \gls{srg} for normal matrices with real entries. We will now introduce the relevant concepts and review these results. The \hyperlink{beltramikleinmapping}{Beltrami-Klein mapping} maps the \hyperlink{complexplane}{complex plane} into the closed unit disk through
\[
f\funof{z}=\frac{\funof{\overline{z}-i}\funof{z-i}}{1+\overline{z}z}.
\]
This mapping sends generalised circles centred on the real axis onto chords of the unit circle, as illustrated in \Cref{fig:bk}. We will also need to apply this function to \hyperlink{boundedlinearopeartor}{linear operators}, in which case it will be understood that
\[
f\funof{T}=\funof{I+T^*T}^{-\frac{1}{2}}\funof{T^*-iI}\funof{T-iI}\funof{I+T^*T}^{-\frac{1}{2}}.
\]
Note that $f\funof{z}$ is not bijective, since $f\funof{z}=f\funof{\overline{z}}$. However, for any $z$ for which $\imag{z}\geq{}0$,
\[
z=\frac{\imag{f\funof{z}}+i\sqrt{1-\abs{f\funof{z}}^2}}{\real{f\funof{z}}-1}.
\]
This relation motivates the definition of
\[
g\funof{z}=\cfunof{\frac{\imag{z}\pm{}i\sqrt{1-\abs{z}^2}}{\real{z}-1}}.
\]
This map sends each point in the closed unit disk back to the corresponding \hyperlink{complexconjugate}{complex conjugate} pair ($g\funof{f\funof{z}}=\cfunof{z,\overline{z}}$). Since $\srg{T}=\overline{\srg{T}}$, this establishes that
\begin{equation}\label{eq:sim}
\srg{T}=g\funof{f\funof{\srg{T}}}.
\end{equation}
As we will see in the next section, $f\funof{\srg{T}}$ is in many ways simpler to understand than \srg{T}. Equation~\eqref{eq:sim} then shows that we can always convert a result on $f\funof{\srg{T}}$ back to a result on $\srg{T}$ using $g\funof{\cdot}$. 

This pattern of obtaining a simplified analysis of $f\funof{\srg{T}}$ can also be seen in the main result of \cite{HRY20}. Introducing the notation $\hullbk{\cdot{}}=g\funof{\hull{f\funof{\cdot}}}$, there it was shown that if $T$ is a matrix with real entries (acting on a real \hyperlink{Hilbertspace}{Hilbert space}) and $T=T^\mathsf{T}$, then $f\funof{\srg{T}}=\hull{f\funof{\spectrum{T}}}$, or equivalently
\[
\srg{T}=\hullbk{\spectrum{T}}.
\]
Furthermore the above remains true for $TT^\mathsf{T}=T^\mathsf{T}T$ with the understanding that \spectrum{T} denotes the \hyperlink{extendedspectrum}{spectrum} of $T$ when viewed on the corresponding complex \hyperlink{Hilbertspace}{Hilbert space} (i.e. look at all the eigenvalues of $T$ in $\C$, not just those in the underlying \hyperlink{realfield}{field} \R{} of the operator).

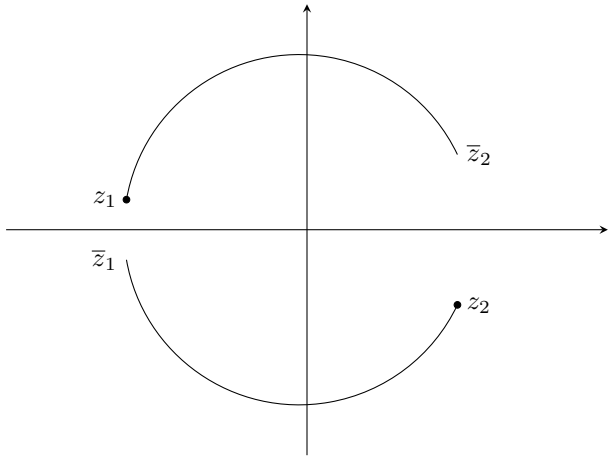
\begin{figure}

\begin{tikzpicture}[every to/.style={hyperbolic plane},scale=2]

\draw (-1.2,0.2) to (1,0.5);
\draw (-1.2,-0.2) to (1,-0.5);

\draw[fill=black] (-1.2,0.2) circle (0.625pt) node[left] {$z_1$};
\draw[fill=black] (1,-.5) circle (0.625pt) node[right] {$z_2$};
\draw (-1.2,-0.2) node[left] {$\overline{z}_1$};
\draw (1,.5) node[right] {$\overline{z}_2$};

\draw[->] (-2,0) -- (2,0);
\draw[->] (0,-1.5) -- (0,1.5);

\end{tikzpicture}
\caption{\label{fig:pc}The hyperbolic straight line between two points $z_1,z_2$ consists of two circular arcs under the Pointcar\'{e} half-plane model.}
\end{figure}

\begin{remark}\label{rem:pc}
Given a set $s\subseteq\C$, $\hullbk{s}$ behaves like the \hyperlink{convexhull}{convex hull}, however with the notion of a straight line taken from hyperbolic geometry under the Poincar\'{e} half-plane model. First recall that $\hull{s}$ is equal to the set of all points that lie on a straight line between $z_1,z_2\in{}s$. In the Poincar\'{e} half-plane model (adapting things slightly for our needs), the straight line between two points $z_1,z_2$ consists of the two arc segments of the generalised circle centred on the real axis that passes through the points $\cfunof{z_1,z_2,\overline{z}_1,\overline{z}_2}$ that:
\begin{enumerate}
  \item connect a pair of points in $\cfunof{z_1,z_2,\overline{z}_1,\overline{z}_2}$;
  \item do not intersect the real axis.
\end{enumerate}
This is illustrated in \Cref{fig:pc}. The set $\hullbk{s}$ is then the set of all points that lie on a  hyperbolic straight line under the Poincar\'{e} half-plane model between $z_1,z_2\in{}s$.
\end{remark}

\subsection{The numerical range}\label{sec:25}

The \hyperlink{numericalrange}{numerical range} is a classical object in the study of linear operators on complex \hyperlink{Hilbertspace}{Hilbert spaces}. For a \hyperlink{boundedlinearoperator}{linear operator} it is defined to be the subset of the \hyperlink{complexplane}{complex plane}
\[
\W{T}=\cfunof{\inner{T\hvect{x}}{\hvect{x}}:\hvect{x}\in\hilb,\norm{\hvect{x}}=1}.
\]
A nice introduction to the \hyperlink{numericalrange}{numerical range}
can be found in \cite{Sha17,HJ91}. The following facts about the \hyperlink{numericalrange}{numerical range} of a \hyperlink{boundedlinearoperator}{linear operator} are standard:
\begin{enumerate}[i)]
	\item $\W{T}=\hull{\W{T}}$ ($\W{T}$ is \hyperlink{convex}{convex});
  \item if $TT^*=T^*T$, then $\Wc{T}=\hull{\spectrum{T}}$ (if $T$ is \hypertarget{normal}{\notation{normal}}, the \hyperlink{closure}{closure} of $\W{T}$ equals the \hyperlink{convexhull}{convex hull} of the \hyperlink{extendedspectrum}{spectrum} of $T$);
	\item $\Wc{T}\supseteq{}\spectrum{T}$ (the \hyperlink{closure}{closure} of $\W{T}$ contains the \hyperlink{extendedspectrum}{spectrum} of $T$).
\end{enumerate}
More generally, the similarity invariance of the \hyperlink{extendedspectrum}{spectrum} implies that the \hyperlink{convexhull}{convex hull} of the \hyperlink{extendedspectrum}{spectrum} is also contained in $\Wc{STS^{-1}}$, and hence in the intersection of the sets $\Wc{STS^{-1}}$ for all choices of $S$. An elegant result of Hildebrant \cite{Hil66} shows that this containment is tight, in the sense that
\begin{enumerate}[i)]\addtocounter{enumi}{3}
	\item $\hull{\spectrum{T}}=\bigcap\cfunof{\Wc{STS^{-1}}:S,S^{-1}\text{ are \hyperlink{boundedlinearoperator}{linear operators}}}.$
\end{enumerate}

\section{Results}\label{sec:3}

\subsection{Connection to the numerical range}\label{sec:31}

In this subsection we connect the \gls{srg} to the \hyperlink{numericalrange}{numerical range}. The following theorem shows that the \gls{srg} of a \hyperlink{boundedlinearoperator}{linear operator} $T$ on a complex \hyperlink{Hilbertspace}{Hilbert space} can be obtained from the \hyperlink{numericalrange}{numerical range} of $f\funof{T}$. Furthermore $\srg{T}$ is endowed with \hyperlink{convex}{convexity} and \hyperlink{extendedspectrum}{spectral} properties along the lines of i)--iv) from \cref{sec:25}, with two main differences.
\begin{enumerate}
  \item The notion of \hyperlink{convex}{convexity} is taken with respect to the Poincar\'{e} half-plane model, as explained in \Cref{rem:pc} (i.e. replace $\hull{\cdot}$ with $\hullbk{\cdot}$).
  \item The \hyperlink{extendedspectrum}{spectral} properites pertain to the \hyperlink{approximatepointspectrum}{approximate point spectrum} instead of the \hyperlink{extendedspectrum}{spectrum} (i.e. replace $\spectrum{T}$ with $\spectrumap{T}$).
\end{enumerate}
This gives the \gls{srg} a similar geometrical flavour to the \hyperlink{numericalrange}{numerical range}, albeit with respect to a different geomtery. The fact that the \gls{srg} lifts out features of the \hyperlink{approximatepointspectrum}{approximate point spectrum} rather than the \hyperlink{extendedspectrum}{spectrum} is curious, but of no consequence if $T$ is finite dimensional or \hyperlink{normal}{normal}, since in these cases $\spectrumap{T}=\spectrum{T}$. In general, as demonstrated by the set of equivalences in the theorem statement, analogues of iii)--iv) also hold for the \hyperlink{extendedspectrum}{spectrum} if and only if $\srgc{T}\supseteq\srgc{T^*}$.
\begin{theorem}\label{thm:1}
Given a \hyperlink{boundedlinearoperator}{linear operator} $T$ on a complex \hyperlink{Hilbertspace}{Hilbert space},
\begin{equation}\label{eq:thm1}
\srg{T}=g\funof{\W{f\funof{T}}}.
\end{equation}
In addition:
\begin{enumerate}[i)]
	\item $\srg{T}=\hullbk{\srg{T}}$;
  \item if $TT^*=T^*T$, then $\srgc{T}=\hullbk{\spectrumap{T}}$;
  \item $\srgc{T}\supseteq\spectrumap{T}$;
  \item $\bigcap\cfunof{\srgc{STS^{-1}}:S,S^{-1}\text{ are \hyperlink{boundedlinearoperator}{linear operators}}}=\hullbk{\spectrumap{T}}$.
\end{enumerate}
Furthermore the following are equivalent:
\begin{enumerate}[i)]\addtocounter{enumi}{4}
	\item $\srgc{T}\supseteq\spectrum{T}$;
	\item $\bigcap\cfunof{\srgc{STS^{-1}}:S,S^{-1}\text{ are \hyperlink{boundedlinearoperator}{linear operators}}}=\hullbk{\spectrum{T}}$;
  \item $\srgc{T}\supseteq{}\srgc{T^*}$;
  \item $\spectrumap{T}\supseteq{}\spectrum{T}\cap\R$.
\end{enumerate}
\end{theorem}
\hyperlink{proof1}{The proof} of this result is given at the end of the subsection after a series of examples.

\begin{example}[The \gls{srg} in the matrix case]
In this example we will illustrate \Cref{thm:1} when the operator $T$ is a matrix with entries in $\C$, and draw some additional conclusions that apply in this case. 
\begin{enumerate}
  \item The \gls{srg} is a compact set ($\srgc{T}=\srg{T}$). This follows directly from the compactness of the \hyperlink{numericalrange}{numerical range} in the finite dimensional case.
  \item The \hyperlink{boundary}{boundary} of $\srg{T}$ is easily computed. This is because $f\funof{T}$ can be computed using standard algorithms, and inner and outer approximations of the \hyperlink{boundary}{boundary} of the \hyperlink{numericalrange}{numerical range} can be computed to arbitrary precision by solving a sequence of eigenvalue problems \cite{HJ91}\footnote{Software for computing the \hyperlink{boundary}{boundary} of the \hyperlink{numericalrange}{numerical range} in the matrix case can be obtained at \href{http://www.ma.man.ac.uk/~higham/mctoolbox}{http://www.ma.man.ac.uk/\textasciitilde{}higham/mctoolbox}.}. This is illustrated in \Cref{fig:ex1a} and \Cref{fig:ex2a}.
  \item The \gls{srg} of $T$ is equal to the \gls{srg} of its \hyperlink{adjoint}{adjoint}. To see this, note that the \hyperlink{approximatepointspectrum}{approximate point spectra} of $T$ and $T^*$ are equal to their \hyperlink{extendedspectrum}{spectra} ($\spectrumap{T}=\spectrum{T}$ and $\spectrumap{T^*}=\spectrum{T^*}$). Therefore statement \textit{v)} in \Cref{thm:1} is true for both $T$ and $T^*$, implying that $\srg{T}=\srg{T^*}$.
  \item $\srg{STS^{-1}}$ can be made arbitrarily close to $\hullbk{\spectrum{T}}$ using a single similarity transform. To see this, note that the Jordan decomposition of $T$ ensures that there exists an \hyperlink{boundedlyinvertible}{invertible} matrix $Q$ such that
  \[
  QTQ^{-1}=D+N,
  \]
  where $D$ is a diagonal matrix consisting of the eigenvalues of $T$, and $N$ is a strictly upper triangular matrix ($N_{jk}=0$ if $j\leq{}k$). Hence if $S_\gamma=\diag{\gamma,\gamma^2,\ldots}Q$, where $\diag{\gamma,\gamma^2,\ldots}$ denotes the diagonal matrix with entries $\gamma,\gamma^2,\ldots{}$, then
  \[
  S_\gamma{}TS_\gamma^{-1}=D+\diag{\gamma,\gamma^2,\ldots}N\diag{\gamma,\gamma^2,\ldots}^{-1}.
  \]
  Since 
  \[
  \lim_{\gamma\rightarrow{}\infty}\norm{\diag{\gamma,\gamma^2,\ldots}N\diag{\gamma,\gamma^2,\ldots}^{-1}}=0
  \]
  and $\srg{D}=\hullbk{\spectrum{T}}$, it follows that by making $\gamma$ sufficiently large the difference between $\srg{S_\gamma{}TS_\gamma^{-1}}$ and $\hullbk{\spectrum{T}}$ can be made arbitrarily small.
\end{enumerate}
\end{example}

\begin{figure}
\subfloat[\label{fig:ex1a}$T:\C^4\rightarrow{}\C^4$, $\spectrum{T}=\cfunof{1,1+\sqrt[3]{2},1-\sqrt[3]{2}\pm\sqrt[6]{27/16}}$.]
{
\begin{tikzpicture}[every to/.style={hyperbolic plane},scale=1.5]

\input{data/example1srg.tex}
\draw[fill=gray!20] (0.37004,1.0911) to (1,0) to (2.2599,0) to cycle;\begin{scope}[yscale=-1] \draw[fill=gray!20] (0.37004,1.0911) to (1,0) to (2.2599,0) to cycle; \end{scope}
\fill (0.37004,1.0911) circle[radius=1pt];\fill (0.37004,-1.0911) circle[radius=1pt];\fill (1,0) circle[radius=1pt];\fill (2.2599,0) circle[radius=1pt];

\draw (-3,0) node {\phantom{$\frac{d^2}{dt^2}\hvect{y}+2\frac{d}{dt}\hvect{y}+\hvect{y}=2\hvect{x}$}};
\draw (-3,0) node {$\hvect{y}=\begin{bmatrix}
1&0&-1&0\\0&2&0&1\\1&1&0&0\\0&0&1&1
\end{bmatrix}\hvect{x}$};

\draw[->] (-1,0) -- (3,0);
\draw[->] (0,-2.25) -- (0,2.25);
\end{tikzpicture}
}

\vspace{.5cm}
\subfloat[\label{fig:ex1b}$T:\Ltwo{}\rightarrow{}\Ltwo$, $\spectrum{T}=\cfunof{2/{\funof{i\omega+1}^2}:\omega\in\R\cup\cfunof{\infty{}}}$.]
{
\begin{tikzpicture}[every to/.style={hyperbolic plane},scale=1.5]

\draw[fill=gray!20] (1.9999,-0.012649) to (1.9999,-0.015627) to (1.9999,-0.019307) to (1.9998,-0.023853) to (1.9997,-0.029469) to (1.9995,-0.036406) to (1.9992,-0.044975) to (1.9988,-0.055558) to (1.9982,-0.068627) to (1.9973,-0.084762) to (1.9959,-0.10467) to (1.9937,-0.12923) to (1.9904,-0.15948) to (1.9854,-0.19671) to (1.9778,-0.24241) to (1.9662,-0.29832) to (1.9486,-0.36638) to (1.9222,-0.44859) to (1.8826,-0.54666) to (1.8238,-0.66151) to (1.738,-0.7921) to (1.6152,-0.93383) to (1.4453,-1.0763) to (1.2209,-1.2013) to (0.94337,-1.2829) to (0.63009,-1.2928) to (0.31678,-1.2121) to (0.049162,-1.0458) to (-0.13643,-0.8258) to (-0.23017,-0.59807) to (-0.24921,-0.40082) to (-0.2233,-0.25187) to (-0.1793,-0.15052) to (-0.13428,-0.086673) to (-0.096078,-0.048601) to (-0.066707,-0.026756) to (-0.045404,-0.014546) to (-0.030501,-0.0078414) to (-0.020314,-0.0042035) to (-0.013453,-0.002245) to (-0.0088764,-0.0011961) to (-0.0058422,-0.00063619) to (-0.0038391,-0.00033803) to (-0.0025201,-0.00017948) to (-0.0016532,-9.5256e-05) to (-0.001084,-5.0539e-05) to (-0.00071054,-2.6809e-05) to (-0.00046566,-1.4219e-05) to (-0.00030514,-7.5412e-06) to (-0.00019994,-3.9992e-06) to cycle;\begin{scope}[yscale=-1] \draw[fill=gray!20] (1.9999,-0.012649) to (1.9999,-0.015627) to (1.9999,-0.019307) to (1.9998,-0.023853) to (1.9997,-0.029469) to (1.9995,-0.036406) to (1.9992,-0.044975) to (1.9988,-0.055558) to (1.9982,-0.068627) to (1.9973,-0.084762) to (1.9959,-0.10467) to (1.9937,-0.12923) to (1.9904,-0.15948) to (1.9854,-0.19671) to (1.9778,-0.24241) to (1.9662,-0.29832) to (1.9486,-0.36638) to (1.9222,-0.44859) to (1.8826,-0.54666) to (1.8238,-0.66151) to (1.738,-0.7921) to (1.6152,-0.93383) to (1.4453,-1.0763) to (1.2209,-1.2013) to (0.94337,-1.2829) to (0.63009,-1.2928) to (0.31678,-1.2121) to (0.049162,-1.0458) to (-0.13643,-0.8258) to (-0.23017,-0.59807) to (-0.24921,-0.40082) to (-0.2233,-0.25187) to (-0.1793,-0.15052) to (-0.13428,-0.086673) to (-0.096078,-0.048601) to (-0.066707,-0.026756) to (-0.045404,-0.014546) to (-0.030501,-0.0078414) to (-0.020314,-0.0042035) to (-0.013453,-0.002245) to (-0.0088764,-0.0011961) to (-0.0058422,-0.00063619) to (-0.0038391,-0.00033803) to (-0.0025201,-0.00017948) to (-0.0016532,-9.5256e-05) to (-0.001084,-5.0539e-05) to (-0.00071054,-2.6809e-05) to (-0.00046566,-1.4219e-05) to (-0.00030514,-7.5412e-06) to (-0.00019994,-3.9992e-06) to cycle; \end{scope}

\draw (-3,0) node {\phantom{$\hvect{y}=\begin{bmatrix}
1&0&-1&0\\0&2&0&1\\1&1&0&0\\0&0&1&1
\end{bmatrix}\hvect{x}$}};
\draw (-3,0) node {$\frac{d^2}{dt^2}\hvect{y}+2\frac{d}{dt}\hvect{y}+\hvect{y}=2\hvect{x}$};

\draw[domain=0:90, smooth, variable=\x, black,line width=2pt] plot ({(cos(2*\x)*(cos(2*\x) - 1))},{sin(4*\x)/2 - sin(2*\x)});
\draw[domain=0:90, smooth, variable=\x, black,line width=2pt] plot ({(cos(2*\x)*(cos(2*\x) - 1))},{-sin(4*\x)/2 + sin(2*\x)});

\draw[->] (-1,0) -- (3,0);
\draw[->] (0,-2.25) -- (0,2.25);
\end{tikzpicture}
}
\caption{\label{fig:ex1}Illustration of $\srgc{T}$ (the orange region), $\hullbk{\spectrum{T}}$ (the grey region), and $\spectrum{T}$ (the black dots or thick black line) for two different operators.}
\end{figure}

\begin{figure}
\subfloat[\label{fig:ex2a}]{
\begin{tikzpicture}[scale=2.4]

\draw[fill=orange!50] (0.72227,-0.6264) to (0.72239,-0.61682) to (0.72225,-0.60567) to (0.72175,-0.59264) to (0.72079,-0.57738) to (0.7192,-0.5595) to (0.71683,-0.53862) to (0.71345,-0.51444) to (0.70889,-0.48677) to (0.70295,-0.45568) to (0.69554,-0.42156) to (0.68667,-0.38516) to (0.67651,-0.34755) to (0.66532,-0.30993) to (0.65347,-0.27343) to (0.6413,-0.23894) to (0.62913,-0.20701) to (0.61717,-0.17787) to (0.60557,-0.1515) to (0.59437,-0.12769) to (0.58359,-0.10619) to (0.57319,-0.086696) to (0.56313,-0.068926) to (0.55336,-0.052626) to (0.54381,-0.037577) to (0.53444,-0.023593) to (0.52519,-0.010521) to (0.51603,0.0017652) to (0.50691,0.013368) to (0.49781,0.02437) to (0.4887,0.03484) to (0.47955,0.044831) to (0.47035,0.054388) to (0.46107,0.063546) to (0.45172,0.072334) to (0.44226,0.080775) to (0.4327,0.088888) to (0.42303,0.096687) to (0.41324,0.10418) to (0.40333,0.11139) to (0.39328,0.1183) to (0.38311,0.12494) to (0.37281,0.1313) to (0.36237,0.13739) to (0.3518,0.1432) to (0.3411,0.14873) to (0.33026,0.154) to (0.3193,0.15899) to (0.3082,0.16371) to (0.29699,0.16815) to (0.28565,0.17231) to (0.27419,0.1762) to (0.26262,0.1798) to (0.25094,0.18311) to (0.23915,0.18614) to (0.22727,0.18887) to (0.21529,0.19131) to (0.20322,0.19346) to (0.19107,0.19531) to (0.17884,0.19685) to (0.16654,0.19809) to (0.15419,0.19902) to (0.14178,0.19965) to (0.12932,0.19996) to (0.11682,0.19996) to (0.1043,0.19965) to (0.091751,0.19901) to (0.07919,0.19807) to (0.066625,0.1968) to (0.054063,0.19521) to (0.041515,0.1933) to (0.02899,0.19108) to (0.016496,0.18853) to (0.0040428,0.18567) to (-0.0083596,0.18248) to (-0.020702,0.17898) to (-0.032976,0.17516) to (-0.045171,0.17103) to (-0.057277,0.16658) to (-0.069287,0.16183) to (-0.081191,0.15677) to (-0.092979,0.1514) to (-0.10464,0.14573) to (-0.11617,0.13976) to (-0.12756,0.1335) to (-0.1388,0.12695) to (-0.14988,0.12011) to (-0.1608,0.11299) to (-0.17154,0.10559) to (-0.1821,0.097921) to (-0.19247,0.089981) to (-0.20264,0.081776) to (-0.21261,0.073313) to (-0.22238,0.064595) to (-0.23193,0.055628) to (-0.24126,0.046416) to (-0.25036,0.036963) to (-0.25923,0.027275) to (-0.26787,0.017355) to (-0.27626,0.0072077) to (-0.28441,-0.0031632) to (-0.29231,-0.013754) to (-0.29996,-0.024562) to (-0.30735,-0.035584) to (-0.31447,-0.046817) to (-0.32134,-0.058261) to (-0.32794,-0.069914) to (-0.33426,-0.081774) to (-0.34032,-0.093843) to (-0.3461,-0.10612) to (-0.35159,-0.11861) to (-0.35681,-0.13131) to (-0.36174,-0.14423) to (-0.36638,-0.15738) to (-0.37072,-0.17075) to (-0.37477,-0.18435) to (-0.37851,-0.1982) to (-0.38194,-0.21231) to (-0.38506,-0.22668) to (-0.38785,-0.24133) to (-0.39031,-0.25627) to (-0.39244,-0.27153) to (-0.3942,-0.28711) to (-0.39561,-0.30304) to (-0.39663,-0.31935) to (-0.39726,-0.33604) to (-0.39748,-0.35315) to (-0.39726,-0.3707) to (-0.39658,-0.38871) to (-0.39541,-0.4072) to (-0.39374,-0.42619) to (-0.39152,-0.44568) to (-0.38873,-0.4657) to (-0.38535,-0.48623) to (-0.38134,-0.50725) to (-0.37667,-0.52874) to (-0.37134,-0.55065) to (-0.36532,-0.57291) to (-0.35862,-0.59545) to (-0.35125,-0.61815) to (-0.34322,-0.64088) to (-0.33459,-0.66352) to (-0.3254,-0.6859) to (-0.31573,-0.70786) to (-0.30566,-0.72926) to (-0.29529,-0.74995) to (-0.28471,-0.76978) to (-0.27401,-0.78867) to (-0.26331,-0.80651) to (-0.25269,-0.82326) to (-0.24222,-0.83888) to (-0.23197,-0.85337) to (-0.222,-0.86674) to (-0.21234,-0.87903) to (-0.20304,-0.89028) to (-0.1941,-0.90055) to (-0.18553,-0.90991) to (-0.17734,-0.91842) to (-0.16952,-0.92615) to (-0.16205,-0.93316) to (-0.15493,-0.93952) to (-0.14813,-0.9453) to (-0.14163,-0.95054) to (-0.13542,-0.95529) to (-0.12947,-0.95962) to (-0.12376,-0.96355) to (-0.11827,-0.96713) to (-0.11298,-0.9704) to (-0.10786,-0.97338) to (-0.10289,-0.97611) to (-0.09806,-0.97862) to (-0.093339,-0.98091) to (-0.088708,-0.98302) to (-0.084146,-0.98496) to (-0.07963,-0.98675) to (-0.075139,-0.9884) to (-0.070648,-0.98992) to (-0.066133,-0.99132) to (-0.061565,-0.99262) to (-0.056916,-0.99381) to (-0.052152,-0.99491) to (-0.047237,-0.99591) to (-0.042128,-0.99682) to (-0.036778,-0.99763) to (-0.031132,-0.99834) to (-0.025124,-0.99895) to (-0.018677,-0.99943) to (-0.011701,-0.99978) to (-0.0040867,-0.99997) to (0.0042974,-0.99997) to (0.013609,-0.99974) to (0.024037,-0.99921) to (0.035812,-0.99832) to (0.049205,-0.99696) to (0.064539,-0.99502) to (0.082185,-0.99233) to (0.10256,-0.98871) to (0.12608,-0.9839) to (0.15318,-0.97766) to (0.18413,-0.9697) to (0.21902,-0.95979) to (0.2576,-0.94778) to (0.29914,-0.9337) to (0.34244,-0.9178) to (0.38598,-0.90056) to (0.42811,-0.88265) to (0.46743,-0.86475) to (0.50293,-0.84751) to (0.53413,-0.83138) to (0.56098,-0.81662) to (0.58377,-0.80335) to (0.60292,-0.79154) to (0.61895,-0.78109) to (0.63234,-0.77187) to (0.64355,-0.76373) to (0.65296,-0.75654) to (0.66089,-0.75015) to (0.6676,-0.74445) to (0.67333,-0.73934) to (0.67824,-0.73473) to (0.68248,-0.73054) to (0.68617,-0.72672) to (0.68939,-0.72321) to (0.69222,-0.71995) to (0.69474,-0.71691) to (0.69698,-0.71406) to (0.69899,-0.71136) to (0.70082,-0.70878) to (0.70247,-0.70631) to (0.70399,-0.70392) to (0.70539,-0.70158) to (0.70669,-0.69929) to (0.70791,-0.69701) to (0.70905,-0.69473) to (0.71013,-0.69244) to (0.71116,-0.69011) to (0.71214,-0.68771) to (0.71308,-0.68524) to (0.71399,-0.68266) to (0.71488,-0.67994) to (0.71573,-0.67706) to (0.71657,-0.67397) to (0.71738,-0.67064) to (0.71816,-0.66701) to (0.71892,-0.66303) to (0.71965,-0.65861) to (0.72033,-0.65369) to (0.72096,-0.64814) to (0.72152,-0.64185) to (0.72197,-0.63467) to cycle;
\draw[fill=gray!20] (0.1407,-0.31797) to (-1.1102e-15,-1) to (0.67252,-0.74008) to cycle;
\fill (0.1407,-0.31797) circle[radius=.625pt];\fill (0.1407,-0.31797) circle[radius=.625pt];\fill (-1.1102e-15,-1) circle[radius=.625pt];\fill (0.67252,-0.74008) circle[radius=.625pt];

\draw(0,0) circle (1);
\draw[->] (-1.25,0) -- (1.25,0);
\draw[->] (0,-1.25) -- (0,1.25);

\end{tikzpicture}}\hfill{}
\subfloat[\label{fig:ex2b}]{
\begin{tikzpicture}[scale=2.4]

\draw[fill=gray!20] (0.59999,-0.79999) to (0.59999,-0.79998) to (0.59999,-0.79997) to (0.59998,-0.79996) to (0.59997,-0.79994) to (0.59995,-0.79991) to (0.59992,-0.79986) to (0.59988,-0.79978) to (0.59981,-0.79967) to (0.59971,-0.7995) to (0.59956,-0.79923) to (0.59933,-0.79883) to (0.59898,-0.79821) to (0.59844,-0.79727) to (0.59761,-0.79583) to (0.59636,-0.79364) to (0.59444,-0.79029) to (0.5915,-0.7852) to (0.58702,-0.77745) to (0.58017,-0.76568) to (0.5697,-0.74784) to (0.55366,-0.72092) to (0.52912,-0.68057) to (0.49157,-0.62073) to (0.43435,-0.53361) to (0.34816,-0.41072) to (0.22168,-0.24656) to (0.045856,-0.046908) to (-0.17608,0.16046) to (-0.41777,0.32632) to (-0.63564,0.40762) to (-0.79646,0.40115) to (-0.89608,0.33997) to (-0.95019,0.26187) to (-0.97708,0.18995) to (-0.98972,0.13273) to (-0.99546,0.090601) to (-0.99802,0.060943) to (-0.99914,0.040611) to (-0.99963,0.026902) to (-0.99984,0.017751) to (-0.99993,0.011684) to (-0.99997,0.0076781) to (-0.99999,0.0050402) to (-0.99999,0.0033063) to (-1,0.0021679) to (-1,0.0014211) to (-1,0.00093133) to (-1,0.00061029) to (-1,0.00039988) to cycle;
\draw[domain=0:90, smooth, variable=\x, black,line width=2pt] plot ({1-2/(cos(2*\x)^2-2*cos(2*\x)+2)},{-2*(cos(2*\x)*(cos(2*\x) - 1))/(cos(2*\x)^2-2*cos(2*\x)+2)});

\draw(0,0) circle (1);
\draw[->] (-1.25,0) -- (1.25,0);
\draw[->] (0,-1.25) -- (0,1.25);

\end{tikzpicture}}
\caption{\label{fig:ex2}The \protect\hyperlink{beltramikleinmapping}{Beltrami-Klein mapping} of the regions from \Cref{fig:ex1}. In both cases $f\funof{\srgc{T}}$ is {\protect\hyperlink{convex}{convex}}, as guaranteed by \Cref{thm:1}\textit{i)}.}
\end{figure}

\begin{example}[The \gls{srg} in the differential equation case]\label{ex:2}
In this example we study the \gls{srg} of an operator defined by a differential equation. This example can be viewed as a generalisation of \cite[Theorem 1]{CFS21}. In the following\footnote{Throughout this example we will tacitly assume that $\det\funof{z^p+\ldots{}+\alpha_{p-1}z+\alpha_p}\neq{}0$ for all $z\in{}i\R$, and that $p\geq{}q$. We note however that these requirements can be removed by extending \Cref{thm:1} to cover densely defined closed operators (which need not be bounded). Indeed this is not too hard to do since $f\funof{T}$ is still a bounded \hyperlink{boundedlinearoperator}{linear operator} whenever $T$ is a densely defined closed operator. However describing these extensions requires a considerably more densely defined notation, so we will not pursue this further here. A good introduction to densely defined closed operators can be found in \cite[Chapter X]{Con94}.} we will consider
\begin{equation}\label{eq:de}
\tfrac{d^p}{dt^p}\hvect{y}+\ldots{}+\alpha_{p-1}\tfrac{d}{dt}\hvect{y}+\alpha_p\hvect{y}=\beta_0\tfrac{d^q}{dt^q}\hvect{x}+\ldots{}+\beta_{q-1}\tfrac{d}{dt}\hvect{x}+\beta_q\hvect{x},
\end{equation}
where $\alpha_j,\beta_k\in\C$ and $\hvect{x},\hvect{y}\in\Ltwo$, though the approach we describe works just as well when these coefficients are square matrices and $\hvect{x},\hvect{y}$ are vectors of functions in $\Ltwo$. Note that in applications it might seem more natural to work on a real \hyperlink{Hilbertspace}{Hilbert space}, where $\alpha_j,\beta_k\in\R$, and $\hvect{x},\hvect{y}$ are real valued functions. In the next subsection it will be shown that from the perspective of the \gls{srg} this distinction is unimportant, and we may as well consider the case of complex \hyperlink{Hilbertspace}{Hilbert spaces}. 

It is possible to associate a range of different operators $\hvect{x}\mapsto{}\hvect{y}$ with \cref{eq:de} depending on the time interval or the boundary conditions that are being studied. A perspective that has been particularly profitable both in theory and in practice has been to associate \cref{eq:de} with a \hyperlink{boundedlinearoperator}{linear operator} $T:\Ltwo\rightarrow{}\Ltwo$ defined through a multiplication operator $T_h:\Ltwo\rightarrow{}\Ltwo$ in the frequency domain. In this setting, denoting the Fourier transform
 as $F:\Ltwo\rightarrow{}\Ltwo{}$, $T=F^*T_hF$, where
\[
T_h\hat{\hvect{x}}\w=h\w\hat{\hvect{x}}\w,
\]
and
\[
h\w=\frac{\beta_0\funof{i\omega}^q+\ldots{}\beta_{q-1}i\omega+\beta_q}{\funof{i\omega}^p+\ldots{}\alpha_{p-1}i\omega+\alpha_p}.
\]
The function $h\w$ is often referred to as a multiplier or transfer function. We will now show how to determine $\srgc{T}$. The first thing to note is that both the \gls{srg} and the \hyperlink{numericalrange}{numerical range} are unitarily invariant. That is given any \hyperlink{boundedlinearoperator}{linear operator} $U$ such that $UU^*=U^*U=I$, $\srg{U^*TU}=\srg{T}$ and $\W{U^*TU}=\W{T}$. Therefore
\begin{equation}\label{eq:intabove}
\srg{T}=\srg{T_h}=g\funof{\W{S^*\funof{T_h^*-iI}\funof{T_h-iI}S}},
\end{equation}
where $S$ is any \hyperlink{invertible}{invertible} \hyperlink{boundedlinearoperator}{linear operator} such that $SS^*=\funof{I+T_h^*T_h}^{-1}$. The first equality follows from the properties of the Fourier transform, and to see the second, observe that
\[
S^{-1}\funof{I+T^*_hT_h}^{-\frac{1}{2}}
\]
is unitary, and compare \cref{eq:intabove} with the definition of $f\funof{\cdot}$ from \cref{sec:23}. A suitable $S$ can then be obtained by applying factorisation techniques for rational functions. More specifically, the process of spectral factorisation can be used to find a bounded rational function $s:\R\rightarrow{}\R$ such that for all $\omega\in{}\R$,
\[
\frac{1}{\overline{h\w{}}h\w+1}=s\w\overline{s\w{}}.
\]
Such a factorisation is always possible, and can be obtained directly from $h\w$ using a normalised coprime factorisation \cite{Vid85}. For example, if $h\w=2/\funof{i\omega{}+1}^2$ (as in \Cref{fig:ex1b}), then a suitable $s\w$ is given by
\[
s\w=\frac{\funof{i\omega+1}^2}{\funof{iw}^2+\sqrt{2+2\sqrt{5}}\,i\omega+\sqrt{5}}.
\]
The multiplication operator
\[
T_s\hat{\hvect{v}}\w=s\w\hat{\hvect{v}}\w
\]
then satisfies $T_sT_s^*=\funof{I+T_h^*T_h}^{-1}$, and therefore
\[
\srgc{T}=g\funof{\cfunof{\W{\overline{s\w}\funof{\overline{h\w}-i}\funof{h\w-i}s\w}:\omega\in\R\cup\cfunof{\infty}}}.
\]
This is illustrated in \Cref{fig:ex1b} and \Cref{fig:ex2b}. The above process is easily generalised to the case that $\alpha_j,\beta_k$ are square matrices ($h\w$ becomes a matrix of rational functions, and $s\w$ can be obtained through the process of normalised right coprime factorisation). Note that in this setting $T$ is not guaranteed to be \hyperlink{normal}{normal}, and so unlike in the case of scalar coefficients $\srgc{T}$ is not necessarily equal to $\hullbk{\spectrum{T}}$.
\end{example}

\begin{example}[The \gls{srg} of the right shift operator]

We have now seen two examples of operators for which the statements \textit{v)--viii)} in \Cref{thm:1} were true, and the \gls{srg} gave information on both the \hyperlink{approximatepointspectrum}{approximate point spectrum} and the \hyperlink{extendedspectrum}{spectrum}. We will now study the \gls{srg} of an operator for which this is not the case. To this end, consider the right shift operator $T:\ltwo\rightarrow{}\ltwo$ given by
\[
T\funof{\hvect{x}_1,\hvect{x}_2,\ldots}=\funof{0,\hvect{x}_1,\hvect{x}_2,\ldots{}}.
\]
The \hyperlink{adjoint}{adjoint} of $T$ is the left shift operator $\funof{\hvect{x}_1,\hvect{x}_2,\ldots}\mapsto\funof{\hvect{x}_2,\hvect{x}_3,\ldots{}}$. It is possible to compute $\srg{T}$ and $\srg{T^*}$ directly. The steps for $T$ are particularly simple since $T^*T=I$, from which it follows that
\[
\begin{aligned}
f\funof{\srg{T}}&=\frac{1}{2i}\W{T+T^*}\\
&=\cfunof{z:z\in\C,\real{z}=0,\abs{z}<1}.
\end{aligned}
\]
Applying the function $g\funof{\cdot}$ from \cref{sec:23} then shows that 
\[
\srg{T}=\cfunof{z:z\in\C:\abs{z}=1,\real{z}\neq{}0}.
\]
A similar but slightly more involved calculation shows that
\[
\srg{T^*}=\cfunof{z:z\in\C:\abs{z}<1,\real{z}\neq{}0}.
\]
We therefore see that $\srgc{T^*}\supset\srgc{T}$. It then follows that the statements \textit{v)--viii)} in \Cref{thm:1} are false for $T$, but true for $T^*$. This means for example that $\srgc{T}\not\supseteq{}\hullbk{\spectrum{T}}$, but $\srgc{T^*}\supseteq{}\hullbk{\spectrum{T}}$. This is easily confirmed directly (in fact $\spectrumap{T}$ is the unit circle and $\spectrum{T}$ is the closed unit disk, meaning that $\srgc{T}=\spectrumap{T}$ and $\srgc{T^*}=\spectrum{T}$).
\end{example}

We now give the proof of \Cref{thm:1}.

\begin{proof}
\hypertarget{proof1}{}We start by establishing \cref{eq:thm1}. First note that considering the polar representation of a complex number $z=r\exp\funof{i\theta}$ shows that 
\[
f\funof{r\exp\funof{i\theta}}=\frac{r^2-1-2ir\cos{\theta}}{1+r^2}.
\]
In light of our discussion from \cref{sec:23} (c.f. \cref{eq:polar}), we then see that for any $\funof{\hvect{x},\hvect{y}}\in\graph{T}$,
\begin{align}
\label{eq:thm111}\!\!\!\!\!\!\!\!\quad{}f\funof{\frac{\norm{\hvect{y}}}{\norm{\hvect{x}}}\exp\funof{\!\pm{}i\arccos\funof{\frac{\real{\inner{\hvect{y}}{\hvect{x}}}}{\norm{\hvect{y}}\norm{\hvect{x}}}}\!\!}\!\!}
&=
\frac{\frac{\norm{\hvect{y}}^2}{\norm{\hvect{x}}^2}-1-2i\frac{\real{\inner{\hvect{y}}{\hvect{x}}}}{\norm{\hvect{x}}^2}}{1+\frac{\norm{\hvect{y}}^2}{\norm{\hvect{x}}^2}},\\
&=\frac{\norm{\hvect{y}}^2-\norm{\hvect{x}}^2-i\funof{\inner{\hvect{y}}{\hvect{x}}+\inner{\hvect{x}}{\hvect{y}}}}{\norm{\hvect{x}}^2+\norm{\hvect{y}}^2},\nonumber{}\\
&=\frac{\inner{R\funof{\hvect{x},\hvect{y}}}{\funof{\hvect{x},\hvect{y}}}}{\norm{\hvect{x}}^2+\norm{\hvect{y}}^2}\nonumber{},
\end{align}
where $R\funof{\hvect{x},\hvect{y}}=\funof{-i\hvect{y}-\hvect{x},\hvect{y}-i\hvect{x}}$. Consider now the linear map
\[
U\hvect{v}=\funof{\funof{I+T^*T}^{-\frac{1}{2}}\hvect{v},T\funof{I+T^*T}^{-\frac{1}{2}}\hvect{v}}.
\]
It is easily checked that for all $\hvect{v}\in\hilb$, $\inner{U\hvect{v}}{U\hvect{v}}=\inner{\hvect{v}}{\hvect{v}}$, $\graph{T}=\cfunof{U\hvect{v}:\hvect{v}\in\hilb}$, and
\[
\inner{RU\hvect{v}}{U\hvect{v}}=\inner{f\funof{T}\hvect{v}}{\hvect{v}}.
\]
Therefore $f\funof{\srg{T}}=\W{f\funof{T}}$, which shows \cref{eq:thm1}. Point \textit{i)} is then immediate from the Toeplitz-Hausdorff theorem. 

\begin{figure}
\subfloat[\label{fig:proofa}]
{
\begin{tikzpicture}[every to/.style={hyperbolic plane},scale=1.5]

\let\radius\undefined
\newlength{\radius}
\setlength{\radius}{1pt}

\draw[fill=orange!50,even odd rule] (-1,0) circle (1.8*221/220) (-1,0) circle (1);

\coordinate (a) at (-.2,.6);
\coordinate (a1) at (-.2,-.6);
\coordinate (b) at (.2,1);
\coordinate (b1) at (.2,-1);
\coordinate (c) at (.8,-21*1.8/220);
\coordinate (c1) at (.8,21*1.8/220);
\coordinate (alph) at (-1,0);

\draw[fill=gray!20] (a) to (b) to (c1) to cycle;
\draw[fill=gray!20] (a1) to (b1) to (c) to cycle;

\fill (a) circle[radius=\radius] node[left] {$p_1$};
\fill (b) circle[radius=\radius] node[above] {$p_2$};
\fill (c) circle[radius=\radius] node[right] {$p_3$};
\fill (alph) circle[radius=\radius];

\draw[->] (-3,0) -- (1,0);
\draw[->] (0,-2) -- (0,2);
\draw (-1,.05) -- (-1,-.15);
\draw[->] (0,-.1) -- node[below] {$\alpha$} (-1,-.1);
\draw[->] (-1,0) -- node[above,rotate=-45] {$\ri{\alpha,s}$} ({-1-cos(45)},{sin(45)});
\draw[->] (-1,0) -- node[above,rotate=45] {$\ro{\alpha,s}$} ({-1-221*1.8/220*cos(45)},{-221*1.8/220*sin(45)});
\end{tikzpicture}}\hfill{}
\subfloat[\label{fig:proofb}]{
\begin{tikzpicture}[scale=2.4]

\let\radius\undefined
\newlength{\radius}
\setlength{\radius}{.625pt}

\begin{pgfonlayer}{bg}
\fill[orange!50] (0,0) circle (1);

\begin{scope}
\clip (0,0) circle (1);
\draw[fill=white] ({bkx(0,0)},{bky(0,0)}) -- ({bkx(-2,0)},{bky(-2,0)}) -- (1,2) -- (-2,0);
\draw[fill=white] ({bkx(-1+1.8*221/220,0)},{bky(-1+1.8*221/220,0)}) -- ({bkx(-1-1.8*221/220,0)},{bky(-1-1.8*221/220,0)})-- (2,0) -- (-1,-2);
\end{scope}

\draw(0,0) circle (1);

\end{pgfonlayer}{bg}

\coordinate (a) at ({bkx(-.2,.6)},{bky(-.2,.6)});
\coordinate (b) at ({bkx(-.2,1)},{bky(-.2,1)});
\coordinate (c) at ({bkx(.8,-21*1.8/220)},{bky(.8,-21*1.8/220)});
\coordinate (alph) at ({bkx(-1,0)},{bky(-1,0)});

\draw[fill=gray!20] (a) to (b) to (c) to (a);
\draw[->] (-1.25,0) -- (1.25,0);
\draw[->] (0,-1.25) -- (0,1.25);

\fill (a) circle[radius=\radius] node[above] {$f\funof{p_1}$};
\fill (b) circle[radius=\radius] node[right] {$f\funof{p_2}$};
\fill (c) circle[radius=\radius] node[below] {$f\funof{p_3}$};
\fill (alph) circle[radius=\radius] node[above right] {$f\funof{\alpha}$};

\end{tikzpicture}}
\caption{\label{fig:proof} Characterisation of the points $\gamma\in{}\hullbk{s}$, for an example with $s=\cfunof{p_1,p_2,p_3}$.}
\end{figure}

We will now show \textit{ii)--iv)}. First denote the shortest and longest distances from a point $\gamma\in\C$ to a set $s\subseteq{}\C$ as
\[
\ri{\gamma,s}=\inf\cfunof{\abs{z-\gamma}:z\in{}s}\;\text{and}\;\ro{\gamma,s}=\sup\cfunof{\abs{z-\gamma}:z\in{}s}
\]
respectively. We will start by showing that given any $s\subseteq{}\C$,
\begin{equation}\label{eq:cohl}
\cl{\hullbk{s}}=\bigcap_{\alpha\in\R}\cfunof{z:z\in\C,\ri{\alpha,s}\leq{}\abs{z-\alpha}\leq\ro{\alpha,s}}.
\end{equation}
To see this, observe that for any value of $\alpha\in\R$, the inequalities in \cref{eq:cohl} characterise the points that lie outside a circle centred on $\alpha$ with radius $\ri{\alpha,s}$ and lie inside a circle centred on $\alpha$ with radius $\ro{\alpha,s}$. This is illustrated in \Cref{fig:proofa}, and the region in question corresponds to the orange annulus. \Cref{fig:proofb} shows the \hyperlink{beltramikleinmapping}{Beltrami-Klein mapping} of these regions. Since the \hyperlink{beltramikleinmapping}{Beltrami-Klein mapping} bijectively maps circles centred on the real axis to chords of the unit circle, and $f\funof{\hullbk{s}}$ is \hyperlink{convex}{convex}, this annulus contains $\cl{\hullbk{s}}$. Conversely every supporting hyperplane for the set $f\funof{\hullbk{s}}$ corresponds to a circle centred on some value of $\alpha\in\R$, and so the intersection of these regions gives $\cl{\hullbk{s}}$.

Next note that $\srg{T-\alpha{}I}=\srg{T}-\alpha$. It then follows from the definition of the \gls{srg} that
\begin{align}
\label{eq:out5}\ri{\alpha,\srg{T}}&=\innerm{T-\alpha{}I},\\
\label{eq:out6}\ro{\alpha,\srg{T}}&=\norm{T-\alpha{I}},
\end{align}
where in the first equation we have introduced the notation
\[
\innerm{A}=\inf\cfunof{\norm{A\hvect{x}}:\hvect{x}\in\hilb,\norm{\hvect{x}}=1}.
\]
It is then easily shown that
\begin{align}
\label{eq:out3}\innerm{T-\alpha{}I}&\leq{}\ri{\alpha,\spectrumap{T}},\\
\label{eq:out4}\norm{T-\alpha{}I}&\geq{}\ro{\alpha,\spectrumap{T}}.
\end{align}
The second of these inequalities is most usually stated in terms of the spectral radius (i.e. replace $\spectrumap{\cdot}$ with $\spectrum{\cdot}$). However, as shown in \cite[Problem 63]{Hal12}, $\partial{}\spectrum{T}\subseteq\spectrumap{T}$ and so this substitution incurs no loss. We therefore see that 
\[
\ri{\alpha,\srg{T}}\leq{}\ri{\alpha,\spectrumap{T}}\;\text{and}\;\ro{\alpha,\srg{T}}\geq{}\ro{\alpha,\spectrumap{T}}.
\]
When combined with \cref{eq:cohl} this shows that $\hullbk{\spectrumap{T}}\subseteq{}\srgc{T}$ (the \hyperlink{approximatepointspectrum}{approximate point spectrum} is always a closed set), which shows \textit{iii)}. This claim can be strengthened to an equality whenever \cref{eq:out3,eq:out4} are equalities for all $\alpha\in\R$. This is the case if $TT^*=T^*T$, which shows \textit{ii)}. To show \textit{iv)} we are required to show that if $\gamma\notin\hullbk{\spectrumap{T}}$, then there there exists an \hyperlink{boundedlyinvertible}{invertible} \hyperlink{boundedlinearoperator}{linear operator} $S$ such that $\gamma\notin\srgc{STS^{-1}}$. In light of \cref{eq:out3,eq:out4,eq:out5,eq:out6} this is equivalent to showing that given any $\varepsilon>0$, there exists an \hyperlink{boundedlyinvertible}{invertible} \hyperlink{boundedlinearoperator}{linear operator} $S_1$ such that
\begin{equation}
\label{eq:out7}
\innerm{S_1\funof{T-\alpha{}I}S_1^{-1}}>\ri{\alpha,\spectrumap{T}}-\varepsilon
\end{equation}
and there exists an \hyperlink{boundedlyinvertible}{invertible} \hyperlink{boundedlinearoperator}{linear operator} $S_2$ such that
\begin{equation}
\label{eq:out8}
\norm{S_2\funof{T-\alpha{}I}S_2^{-1}}<\ro{\alpha,\spectrumap{T}}+\varepsilon.
\end{equation}
In fact \cref{eq:out8} is a well known consequence of Rota's theorem \cite{Rot60}, so we will only show \cref{eq:out7}. By \cite[Theorem 1]{MJ83}, 
\[
\lim_{n\rightarrow{}\infty}\innerm{\funof{T-\alpha{}I}^n}^{\frac{1}{n}}=\ri{\alpha,\spectrumap{T}}.
\]
Therefore there exists a natural number $n$ such that 
\[
\innerm{\funof{T-\alpha{}I}^n}^{\frac{1}{n}}>\ri{\alpha,\spectrumap{T}}-\varepsilon.
\]
Now let
\[
A=\frac{1}{\ri{\alpha,\spectrumap{T}}-\varepsilon}\funof{T-\alpha{}I},
\]
and note that $\innerm{A^n}>1$. Defining $X=I+A^*A+\ldots{}+\funof{A^{n-1}}^*A^{n-1}$ we then see that for any non-zero $\hvect{x}\in\hilb$,
\[
\inner{\funof{A^*XA-X}\hvect{x}}{\hvect{x}}=\inner{\funof{\funof{A^{n}}^*A^{n}-I}\hvect{x}}{\hvect{x}}\geq{}\funof{\innerm{A^n}^2-1}\norm{\hvect{x}}^2>0.
\]
Furthermore since $\inner{X\hvect{x}}{\hvect{x}}\geq{}\norm{\hvect{x}}^2$, there exists an \hyperlink{boundedlyinvertible}{invertible} \hyperlink{boundedlinearoperator}{linear operator} $S_1$ such that $X=S_1^*S_1$. Putting $S_1\hvect{x}=\hvect{y}$ we now see that
\[
\frac{\inner{\funof{A^*XA-X}\hvect{x}}{\hvect{x}}}{\inner{S_1\hvect{x}}{S_1\hvect{x}}}=\frac{\inner{S_1AS_1^{-1}\hvect{y}}{S_1AS_1^{-1}\hvect{y}}}{\norm{\hvect{y}}^2}-1>0.
\]
Therefore $\innerm{S_1AS_1^{-1}}>1$, and so \cref{eq:out7} holds.

To complete the proof we focus on the equivalence of \textit{v)--viii)}.

\textit{vii)}$\,\Rightarrow{}$\textit{v)}: First note that $\spectrum{T}\subseteq\spectrumap{T}\cup\overline{\spectrumap{T^*}}$. Since by definition $\srg{T}=\overline{\srg{T}}$, this shows that $\spectrum{T}\subseteq{}\srgc{T}\cup\srgc{T^*}$, and so by the hypothesis of \textit{vii)} $\spectrum{T}\subseteq{}\srgc{T}$.

\textit{v)}$\,\Rightarrow{}$\textit{viii)}: We proceed by contraposition. Assume that $\alpha\in\spectrum{T}\cap\R$ is not in $\spectrumap{T}$, and so $\innerm{T-\alpha{}I}>0$. From the definition of the \gls{srg}, this imples that $0\notin\srgc{T-\alpha{}I}$. Hence $\alpha\notin\srgc{T}$, and so $\spectrum{T}\not\subseteq{}\srgc{T}$ as required.

\textit{viii})$\,\Rightarrow{}$\textit{vii)}: First note that $\norm{T-\alpha{}I}=\norm{T^*-\alpha{}I}$, and if $\alpha\notin{}\spectrum{T}$, then
\[
\innerm{T-\alpha{}I}=1/\Vert\funof{T-\alpha{}I}^{-1}\Vert=1/\Vert\funof{T^*-\alpha{}I}^{-1}\Vert.
\]
Consider again \cref{eq:out5,eq:out6}. Observe in particular that given any $\alpha\in\R$, under the hypothesis of \textit{viii}) $\ri{\alpha,\srg{T}}\neq{}0$ only if $\alpha\notin\spectrum{T}$. We therefore see from \cref{eq:cohl} that $\gamma\notin\srgc{T}$ only if $\gamma\notin\srgc{T^*}$ as required.

\textit{viii})$\,\Rightarrow{}$\textit{vi)}: Recall that $\partial\spectrum{T}\subseteq{}\spectrumap{T}$. Therefore under the hypothesis of \textit{viii}), if $\alpha\in\R$, then $\ri{\alpha,\spectrumap{T}}=\ri{\alpha,\spectrum{T}}$, and so $\hullbk{\spectrum{T}}=\hullbk{\spectrumap{T}}$. $\textit{vi)}$ now follows from \textit{iv)}.

\textit{vi)}$\,\Rightarrow{}$\textit{v)}: Immediate.
\end{proof}

\subsection{Real Hilbert spaces}\label{sec:32}

In the previous subsection we showed that for a \hyperlink{boundedlinearoperator}{linear operator} acting on a complex \hyperlink{Hilbertspace}{Hilbert space}, the concept of the \gls{srg} is closely related to the \hyperlink{numericalrange}{numerical range}. However, largely motivated by applications from convex optimization, the \gls{srg} has primarily been studied in the context of \hyperlink{Hilbertspace}{Hilbert spaces} over $\R$. At first sight, it might seem like there are fundamental differences between the real and complex case. For example when viewed as an operator on a real \hyperlink{Hilbertspace}{Hilbert space} with \hyperlink{innerproduct}{inner product} $\inner{\hvect{y}}{\hvect{x}}=\hvect{x}^\mathsf{T}\hvect{y}$, 
\begin{equation}\label{eq:2by2}
f\funof{\srg{\begin{bmatrix}
0&1\\0&0
\end{bmatrix}}}=\cfunof{z:\abs{z+\frac{1+i}{2}}+\abs{z+\frac{1-i}{2}}=\sqrt{2},z\in\C}.
\end{equation}
This is not a \hyperlink{convex}{convex} set (it is the \hyperlink{boundary}{boundary} of an ellipse), and therefore \Cref{thm:1} \textit{i)} fails. However when we view the same matrix as an operator on $\C^2$ with \hyperlink{innerproduct}{inner product} $\inner{\hvect{y}}{\hvect{x}}=\overline{\hvect{x}}^\mathsf{T}\hvect{y}$ we obtain
\[
f\funof{\srg{\begin{bmatrix}
0&1\\0&0
\end{bmatrix}}}=\cfunof{z:\abs{z+\frac{1+i}{2}}+\abs{z+\frac{1-i}{2}}\leq{}\sqrt{2},z\in\C}.
\]
That is the \gls{srg} of the operator on the real \hyperlink{Hilbertspace}{Hilbert space} is equal to the \hyperlink{boundary}{boundary} of the \gls{srg} of its complexified counterpart, suggesting the two objects are in fact closely related. This is illustrated in \Cref{fig:5}.

A similar behaviour is seen when studying tuples of Hermitian forms (of which the \hyperlink{numericalrange}{numerical range} is a special case). More specifically, given two $n\times{}n$ symmetric matrices $A$ and $B$ with real entries, it was shown in \cite{Bri61} that
\[
\cfunof{\hvect{x}^\mathsf{T}A\hvect{x}+i\hvect{x}^\mathsf{T}B\hvect{x}:\hvect{x}^\mathsf{T}\hvect{x}=1,\hvect{x}\in\R^n}=\begin{cases}\partial\W{A+iB}&\text{if $n=2$;}\\
\W{A+iB}&\text{otherwise.}
\end{cases}
\]
This result relates the joint \hyperlink{numericalrange}{numerical range} $\cfunof{\funof{\inner{A\hvect{x}}{\hvect{x}},\inner{B\hvect{x}}{\hvect{x}}}:\hvect{x}\in\hilb,\norm{\hvect{x}}=1}$ of two operators on a finite dimensional real \hyperlink{Hilbertspace}{Hilbert space}, to the \hyperlink{numericalrange}{numerical range} of a related operator acting on a finite dimensional complex \hyperlink{Hilbertspace}{Hilbert space}. Moreover, it shows that the two are different only if the \hyperlink{Hilbertspace}{Hilbert space} has dimension 2, where instead the real case equals the \hyperlink{boundary}{boundary} of the complex case. The main result of this subsection is an adaptation of the above that shows that the \gls{srg} behaves in an analogous manner. Before stating the result, let us first formalise the notion of \hypertarget{complexification}{complexification} beyond the matrix case. The following, which can be found in \cite[Chapter I]{Con94}, gives the suitable notion of the \hyperlink{complexification}{complexification} of a \hyperlink{Hilbertspace}{Hilbert space}.
\begin{lemma}\label{lem:1}
Let $\hilb$ be a real \hyperlink{Hilbertspace}{Hilbert space}. Then there exists a complex \hyperlink{Hilbertspace}{Hilbert space} $\hilb_{\C}$ and a linear map $U:\hilb{}\rightarrow{}\hilb_{\C}$ such that:
\begin{enumerate}[i)]
	\item $\inner{U\hvect{x}_1}{U\hvect{x}_2}=\inner{\hvect{x}_1}{\hvect{x}_2}$ for all $\hvect{x}_1,\hvect{x}_2\in\hilb$;
	\item for any $\hvect{y}\in\hilb_{\C}$, there are unique $\hvect{x}_1,\hvect{x}_2\in\hilb$ such that $\hvect{y}=U\hvect{x}_1+iU\hvect{x}_2$.
\end{enumerate}
\end{lemma}

\begin{figure}
\subfloat[]{
\begin{tikzpicture}[scale=2.4]

\draw[line width=1.25pt,fill=orange!50](0,0.5) circle (.5);
\draw[line width=1.25pt,fill=orange!50](0,-0.5) circle (.5);
\draw[->] (-1.25,0) -- (1.25,0);
\draw[->] (0,-1.25) -- (0,1.25);

\end{tikzpicture}}\hfill{}
\subfloat[]{
\begin{tikzpicture}[scale=2.4]

\draw(0,0) circle (1);
\draw[line width=1.25pt,fill=orange!50] (-0.5,0) ellipse (.5 and {0.5*sqrt(2)});
\draw[->] (-1.25,0) -- (1.25,0);
\draw[->] (0,-1.25) -- (0,1.25);

\end{tikzpicture}}
\caption{\label{fig:5} {\footnotesize \textsc{({\tiny A})}} shows $\srg{T}$ (the black circles) and $\srg{T_{\C}}$ (the orange region) for the matrix in \cref{eq:2by2}. {\footnotesize \textsc{({\tiny B})}} shows the \protect\hyperlink{beltramikleinmapping}{Beltrami-Klein mapping} of these regions.}
\end{figure}

Given an operator $T$ on a real \hyperlink{Hilbertspace}{Hilbert space} $\hilb$, we define the \hyperlink{complexification}{complexification} of $T$ to be the operator $T_{\C}$ on $\hilb_{\C}$ which satisfies
\[
T_{\C}\funof{U\hvect{x}_1+iU\hvect{x}_2}=UT\hvect{x}_1+iUT\hvect{x}_2,\;\text{for all $\hvect{x}_1,\hvect{x}_2\in\hilb$}.
\]
It is easy enough to check that these abstractions behave exactly as expected in the matrix case (and also in going from operators on real valued \hyperlink{squareintegrablefunction}{square integrable functions} to $\Ltwo{}$). With this definition in place we are ready to state the main result of this subsection. The following theorem shows that in all dimensions except 2 (including the infinite dimensional case), $\srg{T}=\srg{T_\C}$. Furthermore in dimension 2, $\srg{T}$ is equal to the \hyperlink{boundary}{boundary} of $\srg{T_{\C}}$. This means that \Cref{fig:ex1,fig:ex2} also show the \glspl{srg} of the corresponding operators when viewed on a real \hyperlink{Hilbertspace}{Hilbert space}, and in all cases, \srg{T} can be obtained from the \hyperlink{numericalrange}{numerical range} of an operator on a complex \hyperlink{Hilbertspace}{Hilbert space}, as described in the previous subsection. 

\begin{theorem}\label{thm:2}
Let $T$ be a \hyperlink{boundedlinearoperator}{linear operator} on a real \hyperlink{Hilbertspace}{Hilbert space}. Then
\[
\srg{T}=\begin{cases}
\partial\srg{T_{\C}}&\text{if $T$ has dimension 2;}\\
\srg{T_{\C}}&\text{otherwise.}
\end{cases}
\]
\end{theorem}
\begin{proof}
Let us first slightly rework the characterisation of $\srg{T}$ from \Cref{thm:1} to make it suitable for operators on real \hyperlink{Hilbertspace}{Hilbert spaces}. The issue is that as written, $f\funof{T}:\hilb\rightarrow{}\hilb_{\C}$, and so we cannot define its \hyperlink{numericalrange}{numerical range}. However the problem is only superficial, and starting from \cref{eq:thm111} it is easily shown that
\[
f\funof{\srg{T}}=\cfunof{\inner{A\hvect{x}}{\hvect{x}}+i\inner{B\hvect{x}}{\hvect{x}}:\hvect{x}\in\hilb,\norm{\hvect{x}}=1},
\]
where
\[
\begin{aligned}
A&=\funof{I+T^*T}^{-\frac{1}{2}}\funof{T^*T-I}\funof{I+T^*T}^{-\frac{1}{2}}\;\text{and}\;\\
B&=-\funof{I+T^*T}^{-\frac{1}{2}}\funof{T+T^*}\funof{I+T^*T}^{-\frac{1}{2}}.
\end{aligned}
\]
Similarly
\[
f\funof{\srg{T_{\C}}}=\cfunof{\inner{A_{\C}\hvect{y}}{\hvect{y}}+i\inner{B_{\C}\hvect{y}}{\hvect{y}}:\hvect{y}\in\hilb_{\C},\norm{\hvect{y}}=1}.
\]
Direct calculation shows that for any $\hvect{y}=U\hvect{x}_1+iU\hvect{x}_2\in\hilb_{\C}$,
\[
\begin{aligned}
\inner{A_{\C}\hvect{y}}{\hvect{y}}&=\inner{UA\hvect{x}_1+iUA\hvect{x}_2}{U\hvect{x}_1+iU\hvect{x}_2},\\
&=\inner{UA\hvect{x}_1}{U\hvect{x}_1}+\inner{UA\hvect{x}_2}{U\hvect{x}_2}+i\funof{\inner{UA\hvect{x}_2}{U\hvect{x}_1}-\inner{UA\hvect{x}_1}{U\hvect{x}_2}},\\
&=\inner{A\hvect{x}_1}{\hvect{x}_1}+\inner{A\hvect{x}_2}{\hvect{x}_2}+i\funof{\inner{A\hvect{x}_2}{\hvect{x}_1}-\inner{A\hvect{x}_1}{\hvect{x}_2}}.
\end{aligned}
\]
Since $A=A^*$ and $\hilb$ is over $\R$, $\inner{A\hvect{x}_2}{\hvect{x}_1}=\inner{A\hvect{x}_1}{\hvect{x}_2}$, and so the imaginary part in the above equals zero. Using a similar argument for $\inner{B_{\C}\hvect{y}}{\hvect{y}}$ therefore shows that
\[
\begin{aligned}
\inner{A_{\C}\hvect{y}}{\hvect{y}}+i\inner{B_{\C}\hvect{y}}{\hvect{y}}&=\inner{A\hvect{x}_1}{\hvect{x}_1}+i\inner{B\hvect{x}_1}{\hvect{x}_1}+\inner{A\hvect{x}_2}{\hvect{x}_2}+i\inner{B\hvect{x}_2}{\hvect{x}_2}\\
&=p_1\norm{\hvect{x}_1}^2+p_2\norm{\hvect{x}_2}^2,
\end{aligned}
\]
where $p_1,p_2\in{}f\funof{\srg{T}}$.  Noting that $\norm{\hvect{y}}^2=\norm{\hvect{x}_1}^2+\norm{\hvect{x}_2}^2$, this implies that $f\funof{\srg{T}}\subseteq{}f\funof{\srg{T_{\C}}}\subseteq\hull{f\funof{\srg{T}}}$. By \cite[Theorem 2]{Leg05}, the joint \hyperlink{numericalrange}{numerical range} of any two Hermitian forms on a real \hyperlink{Hilbertspace}{Hilbert space} is \hyperlink{convex}{convex} unless that \hyperlink{Hilbertspace}{Hilbert space} has dimension 2. Therefore $f\funof{\srg{T}}$ is \hyperlink{convex}{convex} unless $T$ has dimension 2, which establishes the second case in the theorem statement. For the two dimensional case, as noted in \cite{Bri61}, the set
\[
\cfunof{\inner{A\hvect{x}}{\hvect{x}}+i\inner{B\hvect{x}}{\hvect{x}}:\norm{\hvect{x}}=1,\hvect{x}\in\hilb{}}
\]
can only be an ellipse, circle, line or point. Since these shapes all have \hyperlink{convex}{convex} \hyperlink{boundary}{boundaries}, this then implies that $f\funof{\srg{T}}=f\funof{\partial{}\srg{T_{\C}}}$ as required.
\end{proof}

\section{Conclusions}

We have demonstrated that the \gls{srg} of a \hyperlink{boundedlinearoperator}{linear operator} acting on complex \hyperlink{Hilbertspace}{Hilbert space} can be determined from the \hyperlink{numericalrange}{numerical range} of a closely related \hyperlink{boundedlinearoperator}{linear operator}. This was used to show that \hyperlink{beltramikleinmapping}{Beltrami-Klein mapping} of the \gls{srg} is \hyperlink{convex}{convex}, and derive an analogue of Hildebrant's theorem for the \gls{srg}. It was further shown how to re-purpose algorithms developed for the \hyperlink{numericalrange}{numerical range} to plot the \hyperlink{boundary}{boundary} of the \gls{srg} in the matrix and linear differential equation case. Finally these results were extended to operators on real \hyperlink{Hilbertspace}{Hilbert spaces}, where it was shown that the \gls{srg} could be obtained using the results for complex \hyperlink{Hilbertspace}{Hilbert spaces} through the process of \hyperlink{complexification}{complexification}.

\bibliographystyle{amsplain}
\bibliography{references}

\end{document}